\documentclass[12pt]{article}

\usepackage[margin=2cm]{geometry} 
\usepackage[parfill]{parskip}
\usepackage[utf8]{inputenc}
\usepackage[english]{babel}
\usepackage[style=numeric,backref=true,backrefstyle=none,abbreviate=false,urldate=iso,seconds=true]{biblatex}
\usepackage{csquotes}
\usepackage{amssymb,amsmath,amsthm}
\usepackage{url}
\usepackage{multicol}
\usepackage[yyyymmdd,hhmmss]{datetime}
\usepackage[bottom]{footmisc}
\usepackage{hyperref}
\usepackage{color} 
\usepackage{enumitem}
\usepackage[linesnumbered,vlined,boxed,ruled,algonl]{algorithm2e}
\usepackage{orcidlink}
\newcommand{\eqn}[1]{\begin{displaymath} #1 \end{displaymath}}
\newcommand{\neqn}[1]{\begin{equation} #1 \end{equation}}

\newcommand{\floor}[1]{{\left\lfloor #1 \right\rfloor}}

\newcommand{\integral}[4]{\displaystyle\int_{#3}^{#4} \! #1 \, d#2}

\newcommand{\disp}[0]{\displaystyle}
\newcommand{\abs}[1]{\left\vert #1 \right\vert}

\newcommand{\eval}[3]{\left. #1 \right|_{#2}^{#3}}

\newcommand{\set}[1]{{\left\{#1\right\}}}

\newcommand{\defeq}[0]{\overset{\mathrm{def}}{=}}
\newcommand{\seqnum}[1]{\href{https://oeis.org/#1}{#1}}
\newcommand{\quadtext}[1]{\quad \text{#1} \quad}
\newcommand{\qquadtext}[1]{\qquad \text{#1} \qquad}

\allowdisplaybreaks[4]

\newcommand{\currentdate}{\the\year--\twodigit{\the\month}--\twodigit{\the\day}}

\title{Computation of the Totient Summatory Function}
\author{Lucas Augustus Brown \orcidlink{0000-0002-6000-3735}}
\date{\currentdate}

\usepackage{fancyhdr}
\usepackage{lastpage}
\begingroup
    \makeatletter
    \@for\theoremstyle:=definition,remark,plain\do{%
        \expandafter\g@addto@macro\csname th@\theoremstyle\endcsname{%
            \addtolength\thm@preskip\parskip
            }%
        }
\endgroup
\makeatletter
\renewenvironment{proof}[1][\proofname]{\par
  \vspace{-\topsep}
  \pushQED{\qed}%
  \normalfont
  \topsep0pt \partopsep0pt 
  \trivlist
  \item[\hskip\labelsep
        \itshape
    #1\@addpunct{.}]\ignorespaces
}{%
  \popQED\endtrivlist\@endpefalse
  \addvspace{0pt} 
}
\makeatother

\usepackage{thmtools}
\declaretheorem[style=plain]{theorem}

\declaretheorem[sibling=theorem,style=plain]{lemma}

\hypersetup{
    pdftitle={Computation of the Totient Summatory Function},
    pdfauthor={Lucas Augustus Brown},
    pdfsubject={},
    pdfkeywords={},
    colorlinks=true,
    linkcolor=blue,
    urlcolor=blue,
    citecolor=blue,
}

\newcommand{\floordiv}[2]{\floor{\frac{#1}{#2}}}
\newcommand{\dfloordiv}[2]{\floor{\dfrac{#1}{#2}}}
\newcommand{\isqrt}[1]{\floor{\sqrt{#1}}}
\newcommand{\softO}[0]{\widetilde{O}}
\newcommand{\softTheta}[0]{\widetilde{\Theta}}
\SetKw{KwBreak}{break}
\IncMargin{2em}

\usepackage{multirow}

\begin{document}
\maketitle 
\begin{abstract}
An algorithm is devised for computing $\Phi(n) = \phi(1) + \phi(2) + \cdots + \phi(n)$ in time $\softTheta(n^{2/3})$ and space $\softTheta(n^{1/3})$.  The starting point is an existing algorithm based on the Dirichlet hyperbola method and the Mertens function.  The algorithm is then used to compute $\Phi(10^{19}) = 30396355092701331435065976498046398788$.
\end{abstract}

\section{Introduction}

The totient-summatory function,
\eqn{\Phi(n) = \sum_{k=1}^n \phi(n),}
has been computed out to $\Phi(10^{18})=303963550927013314319686824781290348$ (\seqnum{A064018}) with the aid of algorithms that take $\softTheta(n^{2/3})$ time \cite{griff2023,adamant2023}; the implied constants and logarithmic factors are small enough that replicating this computation takes less than a day on a recent computer running a single-threaded program.  Unfortunately, these algorithms all require the simultaneous storage of at least $\Theta(n^{1/2})$ integers, which is pushing the limits of what the typical desktop computer can handle.

The contribution of this paper is to modify one such algorithm \cite[\texttt{totientSummatoryFast1}]{griff2023} to use $\softTheta(n^{1/3})$ space, enabling the computation of
\eqn{\Phi(10^{19}) = 30396355092701331435065976498046398788}
in less than 9 days and 7 gigabytes.

\subsection{Conventions}

The Dirichlet convolution of $f$ and $g$ is denoted by $f*g$.

The letter $\mu$ is used for both the M\"{o}bius function and an array such that $\mu_k = \mu(k)$.

The letter $M$ is used for both the Mertens function and an array such that $M_k = M(k)$.

The letter $\delta$ denotes the identity function for Dirichlet convolution: $\delta(1)=1$, and $\delta(x)=0$ for all other $x$.

The reported time and space complexities count the arithmetic operations used and the integers stored, not bit operations and bits stored.

\subsection{Overview of the paper}

\begin{itemize}
\item In Section \ref{ExistingAlgorithms}, we review some existing algorithms.
\item In Section \ref{SpaceReduction}, we modify one such algorithm to reduce its space complexity from $\Theta(n^{1/2})$ to $\softTheta(n^{1/3})$, culminating in Algorithm \ref{Algo13}.
\item In Section \ref{Analysis}, we analyze Algorithm \ref{Algo13}, concluding in Theorem \ref{Algo13time} that, with its optimal parameter selection, it takes $\Theta\left(n^{2/3} \cdot (\ln(\ln(n)))^{1/3} \right)$ time and $\Theta\left(n^{1/3} \cdot (\ln(\ln(n)))^{2/3} \right)$ space.
\item In Section \ref{ComputationalResults}, we present the results of running Algorithm \ref{Algo13} on various inputs.
\item Finally, Section \ref{SupportingLemmas} contains some supporting lemmas.
\end{itemize}

\section{Existing algorithms} \label{ExistingAlgorithms}

A recent paper by Hirsch, Kessler, and Mendlovic \cite[\S 5.6]{HKM2024} outlines an algorithm for computing $\Phi(n)$ in $\softO(n^{1/2})$ time and space; furthermore, this algorithm has the same space-time tradeoffs as that paper's prime-counting algorithm \cite{HKM2025}, which means that it has a variant that achieves $O(n^{5/9+\varepsilon})$ time and $O(n^{2/9+\varepsilon})$ space for any $\varepsilon>0$.  However, this algorithm and its variants have never been implemented, and the hidden factors are expected to be large enough to make using it noncompetitive for practical $n$.

\subsection{The Mertens-first algorithm}

By applying the Dirichlet hyperbola method to the convolution $\phi = \mu * I$, where $I(x)=x$, and letting $ab=n$, we obtain the formula
\eqn{\Phi(n) = \sum_{x=1}^{a}\sum_{y=1}^{n/x} \mu(x) \, I(y) + \sum_{y=1}^{b}\sum_{x=1}^{n/y} \mu(x) \, I(y) - \sum_{x=1}^{a}\sum_{y=1}^{b} \mu(x) \, I(y)}
\eqn{ = \sum_{x=1}^{a}\sum_{y=1}^{n/x} y \cdot \mu(x) + \sum_{y=1}^{b}\sum_{x=1}^{n/y} y \cdot \mu(x) - \sum_{x=1}^{a}\sum_{y=1}^{b} y \cdot \mu(x)}
\neqn{\Phi(n) = \underbrace{\sum_{x=1}^{a} \mu(x) \cdot \frac{\floordiv{n}{x} \cdot \left(\floordiv{n}{x} + 1\right)}{2}}_{X} + \underbrace{\sum_{y=1}^{b} y \cdot M(n/y)}_{Y} - \underbrace{\frac{b \cdot (b+1)}{2} \cdot M(a)}_{Z}. \label{PhiFormula}}
The labels $X$, $Y$, and $Z$ will be used later.

Suppose that we have an algorithm that can compute an individual value of $M(x)$ in time $\softO(x^c)$, and note that $c < 1$ is available \cite{DR1996}.  Using a sieve to compute the necessary M\"obius values, but otherwise evaluating this formula na\"{i}vely, takes time
\eqn{\softO\left( a + \sum_{x=1}^b \left(\frac{n}{x}\right)^c + a^c \right)}
\eqn{=\softO\left( a + n^c \integral{x^{-c}}{x}{1}{b} + a^c \right)}
\eqn{=\softO\left( a + n^c\frac{b^{1-c}}{1-c} - n^c\frac{1^{1-c}}{1-c} + a^c \right)}
\eqn{=\softO\left( a + n^c b^{1-c} - n^c + a^c \right)}
\eqn{=\softO\left( a + n a^{c-1} - n^c + a^c \right).}
The third term is always dominated by the second, and the fourth is always dominated by the first.
\eqn{=\softO\left( a + n a^{c-1} \right)}
To balance the contributions of the two terms, we take $a = \softO(n^{1/(2-c)})$.

The Del\'{e}glise-Rivat algorithm \cite{DR1996} allows $c=2/3$, and so using it in this algorithm results in a time complexity of $\softO(n^{3/4})$.  The Mertens function can also be computed with the Helfgott-Thompson algorithm \cite{HT23}, which takes $\softO(n^{3/5})$ time.  Evaluating (\ref{PhiFormula}) as described then takes $\softO(n^{5/7})$ time.

This algorithm suffers from the fact that all those Mertens values are computed one-at-a-time and are not given a chance to contribute to each other.  This can be ameliorated by another application of the Dirichlet hyperbola method.  This time, we use $\delta = \mu * 1$ and set $\alpha\beta=n$ to obtain
\eqn{\sum_{k=1}^n \delta(k) = \sum_{x=1}^{\alpha}\sum_{y=1}^{n/x} \mu(x) \cdot 1 + \sum_{y=1}^{\beta}\sum_{x=1}^{n/y} \mu(x) \cdot 1 - \sum_{x=1}^{\alpha}\sum_{y=1}^{\beta} \mu(x) \cdot 1}
\eqn{1 = \sum_{x=1}^{\alpha} \mu(x) \floordiv{n}{x} + \sum_{y=1}^{\beta} M(n/y) - M(\alpha) \floor{\beta}}
\neqn{M(n) = 1 + \floor{\beta} M(\alpha) - \sum_{x=1}^{\alpha} \mu(x) \floordiv{n}{x} - \sum_{y=2}^{\beta} M(n/y). \label{MertensRecursion}}
When evaluating (\ref{PhiFormula}), we need to find $\mu(k)$ for $1 \leq k \leq a$, $M(n/k)$ for $1 \leq k \leq b$, and $M(a)$.

When evaluating (\ref{MertensRecursion}), we need to find $\mu(k)$ for $1 \leq k \leq \alpha$, $M(n/k)$ for $2 \leq k \leq \beta$, and $M(\alpha)$.

Clearly, these work well together: we can sieve $\mu$ up to $a$, accumulate the values along the way to compute $M$ up to $a$, use (\ref{MertensRecursion}) to compute the remaining Mertens values, and then feed all that data into (\ref{PhiFormula}) to compute $\Phi(n)$.  This results in Algorithm \ref{Algo1}, which I call the \emph{Mertens-first algorithm}.  Note that we do \emph{not} take $\alpha = a$: instead, we use $a = \softTheta(n^{2/3})$ and $\alpha = \sqrt{n}$.

\begin{algorithm}[H] \label{Algo1}
\DontPrintSemicolon
\caption{Compute $\Phi(n)$ in $\softTheta(n^{2/3})$ time and $\softTheta(n^{1/2})$ space \cite{griff2023}.}
\KwData{$n \geq 1$}
\KwResult{$\Phi(n)$}
$a \gets \floor{\softTheta(n^{2/3})}$; $b \gets \floor{n/a}$; $X \gets 0$; $Y \gets 0$; $Z \gets 0$; $m \gets 0$; $s \gets \isqrt{n}$ \label{1-p1start}

\lIf{$\isqrt{n} = \floor{n/\isqrt{n}}$}{$s \gets s-1$}

$\chi \gets \floor{n/s}$ \label{1-p1start0}

Prepare a segmented sieve to compute $\mu(k)$ for $1 \leq k \leq a$.

Let $\mu$ and $M$ be arrays indexed from $1$ through $\isqrt{n}$, inclusive. \label{1-muinit}

Let $M^\prime$ be an array indexed from $1$ to $\floor{n/\isqrt{n}}$, inclusive, initialized to all zeros. \label{1-endinit}
\begin{multicols}{2}
\For{$x=1$ \KwTo $a$}{ \label{1-p12loopstart}
    $v \gets \floor{n/x}$
    
    $m \gets m + \mu(x)$
    
    $X \gets X + \mu(x) \cdot \dfrac{v \cdot (v+1)}{2}$ \label{1-11}
    
    \uIf{$x \leq \isqrt{n}$}{ \label{1-p1a}
        $M_x \gets m$
        
        $\mu_x \gets \mu(x)$ \label{1-p1b}
    }
    
    \addtocounter{AlgoLine}{15}
    
    \ElseIf{$x = \chi$}{ \label{1-p2a}
    
        \lIf{$v \neq b$}{$M^\prime_{v} \gets m$} \label{1-17}
        
        \addtocounter{AlgoLine}{2}
    
        $s \gets s-1$
        
        $\chi \gets \floor{n/s}$
    }
    
    \addtocounter{AlgoLine}{5}
    
    \lIf{$x = a$}{$Z \gets m \cdot \dfrac{b \cdot (b+1)}{2}$} \label{1-p2b}
}

\label{1-p1end}

\emph{lines \ref{1-p3start}--\ref{1-p3end} here}

\columnbreak

\vspace*{\fill}

\For{$y=b$ \KwTo $1$}{ \label{1-p3start}
    $v \gets \floor{n/y}$
    
    $m \gets 1 - v + \isqrt{v} \cdot M_{\isqrt{v}}$ \label{1-23}
    
    \For{$x=2$ \KwTo $\isqrt{v}$}{
        $m \gets m - \mu_x \cdot \floor{v/x}$ \label{1-p3mobius}
        
        \uIf{$\floor{v/x} \leq \isqrt{n}$}{
            $m \gets m - M_{\floor{v/x}}$
        }
        \Else{
            $m \gets m - M^\prime_{\floor{n/\floor{v/x}}}$
        }
    }
    $M^\prime_y \gets M^\prime_y + m$
    
    $Y \gets Y + y \cdot M^\prime_y$ \label{1-p3end}
}
\end{multicols}

\KwRet $X + Y - Z$
\end{algorithm}

The purpose of the non-consecutivity of the line numbering is to coordinate line numbers between Algorithms \ref{Algo1}, \ref{Algo7}, \ref{Algo10}, \ref{Algo12}, and \ref{Algo13}.  The variables $X$, $Y$, and $Z$ correspond to the labels $X$, $Y$, and $Z$ in (\ref{PhiFormula}).

Algorithm \ref{Algo1} has four phases:
\begin{enumerate} \addtocounter{enumi}{-1}
\item In the zeroth phase, lines \ref{1-p1start}--\ref{1-endinit} initialize the computation.
\item In the first phase, covered in lines \ref{1-p12loopstart}--\ref{1-p1b}, we sieve the M\"{o}bius function up to $\isqrt{n}$, accumulate its values to compute the Mertens function, save both $\mu$ and $M$, and accumulate terms into $X$.
\item In the second phase, covered in lines \ref{1-p12loopstart}--\ref{1-11} and \ref{1-p2a}--\ref{1-p2b}, we continue the sieve up to $a$.  We continue to accumulate M\"{o}bius values to compute Mertens values, and we continue to accumulate terms into $X$, but we do not save any $\mu$, and only some Mertens values are saved.  As the final act of phase 2, we compute $Z$.  At this point, $X$ and $Z$ are fully evaluated, and nothing has been done about $Y$.
\item In the third phase, lines \ref{1-p3start}--\ref{1-p3end} feed the stored M\"{o}bius and Mertens values into (\ref{MertensRecursion}) to compute the remaining Mertens values in order of increasing argument---that is, we first compute $M(n/b)$, then $M(n/(b-1))$, then ..., and finally $M(n)$.  As each Mertens value is computed, a term from $Y$ becomes available, and we evaluate it accordingly.
\end{enumerate}
Once the third phase is done, $\Phi(n)$ is computed as $X+Y-Z$.

Line \ref{1-17} is gatekept by the condition $v \neq b$.  This is needed to mitigate an overlap in the phases that occurs for some $(a,n)$ pairs.  In such cases, without the gatekeeping, line \ref{1-17} would set $M^\prime_b$ to $M(a)$, which should be its final value, but it then gets modifed in the first iteration through phase 3, which throws things off.  With the condition $v \neq b$ in place, $M^\prime_b$ is not touched until phase 3.

Algorithm \ref{Algo1} takes $\softTheta(n^{2/3})$ time: phases 0--2 combined clearly take $\softTheta(a)$ time, and phase 3 takes time
\eqn{\softTheta \left( \sum_{y=1}^b \left( \isqrt{\frac{n}{y}} - 1 \right) \right)}
\eqn{= \softTheta \left( \integral{ \sqrt{\frac{n}{y}} }{y}{1}{b} - b \right)}
\eqn{= \softTheta \left( 2 \sqrt{n} \left( \sqrt{b} - 1 \right) - b \right)}
\eqn{= \softTheta \left( \frac{n}{\sqrt{a}} \right).}

Algorithm \ref{Algo1} takes $\softTheta(\sqrt{n})$ space: we use three arrays of $\Theta(\sqrt{n})$ elements each to store the M\"{o}bius and Mertens values, the M\"{o}bius sieving consumes $\softO(\sqrt{a})$ space, and everything else fits in $O(1)$ space.

\section{The Mertens-first algorithm in less space} \label{SpaceReduction}

We now reduce Algorithm \ref{Algo1}'s memory usage from $\softTheta(\sqrt{n})$ to $\softTheta(\sqrt[3]{n})$.  The first step is to observe that we can move line \ref{1-p3mobius} into phase 1.  The work done in that line is essentially as follows:

\begin{algorithm}[H] \label{Algo1mu}
\DontPrintSemicolon
\caption{An extract from Algorithm \ref{Algo1}}
\For{$y=b$ \KwTo $1$}{
    \For{$x=2$ \KwTo $\isqrt{n/y}$}{
        $M^\prime_y \gets M^\prime_y - \mu_x \cdot \dfloordiv{n}{yx}$
    }
}
\end{algorithm}

If we can swap the order of the loops, then we will be able to integrate this line into phase 1 and not have to store the M\"{o}bius array.

This extract iterates over all pairs $(y,x)$ such that $1 \leq y \leq b$ and $2 \leq x \leq \sqrt{n/y}$.  The range accessed by $x$ is therefore $2 \leq x \leq \sqrt{n}$, and for each $x$, $y$ ranges over $1 \leq y \leq \min(b, n/x^2)$.  This extract is therefore essentially equivalent to

\begin{algorithm}[H] \label{Algo1mu_redone}
\DontPrintSemicolon
\caption{Algorithm \ref{Algo1mu}, reordered}
\For{$x=2$ \KwTo $\isqrt{n}$}{
    \For{$y=1$ \KwTo $\min(b,\floor{n/x^2})$}{
        $M^\prime_y \gets M^\prime_y - \mu(x) \cdot \dfloordiv{n}{yx}$
    }
}
\end{algorithm}

It is also easy to move line \ref{1-23} into phase 1.  The work this line does is essentially

\begin{algorithm}[H] \label{Algo1M}
\DontPrintSemicolon
\caption{An extract from Algorithm \ref{Algo1}}
\For{$y=b$ \KwTo $1$}{
    $M^\prime_y \gets M^\prime_y + 1 - \floor{n/y} + \isqrt{n/y} \cdot M_{\isqrt{n/y}}$
}
\end{algorithm}

This is essentially equivalent to

\begin{algorithm}[H] \label{Algo4extract_redone}
\DontPrintSemicolon
\caption{Algorithm \ref{Algo1M}, redone}
\For{$x=1$ \KwTo $a$}{
    \If{$\exists y \;\ni\; 1 \leq y \leq b \;\;\&\;\; x=\isqrt{n/y}$}{
        \ForAll{such $y$}{
            $M^\prime_y \gets M^\prime_y + 1 - \floor{n/y} + x \cdot M_x$
        }
    }
}
\end{algorithm}

which we make more precise as

\begin{algorithm}[H] \label{Algo4extract_redone_again}
\DontPrintSemicolon
\caption{Algorithm \ref{Algo1M}, redone again}
$d \gets b$

$\gamma \gets \isqrt{n/d}$

\For{$x=1$ \KwTo $a$}{
    
    \While{$x = \gamma$}{
        $M^\prime_d \gets M^\prime_d + 1 - \floor{n/d} + x \cdot M_x$
        
        $d \gets d - 1$
        
        $\gamma \gets \isqrt{n/d}$
    }
}
\end{algorithm}

Applying these edits to Algorithm \ref{Algo1} yields Algorithm \ref{Algo7}:
\begin{itemize}
\item Lines \ref{1-p1start} and \ref{1-p1start0} have had actions added to them.
\item Line \ref{1-muinit} no longer calls for the existence of the array $\mu$.
\item Line \ref{1-p1b} has been deleted.
\item Lines \ref{7-15}--\ref{7-21} have been inserted.
\item Line \ref{1-23} has been modified.
\item Line \ref{1-p3mobius} has been deleted.
\end{itemize}

\begin{algorithm}[H] \label{Algo7}
\DontPrintSemicolon
\caption{Compute $\Phi(n)$ in $\softTheta(n^{2/3})$ time and $\softTheta(n^{1/2})$ space.}
\KwData{$n \geq 1$}
\KwResult{$\Phi(n)$}
$a \gets \floor{\softTheta(n^{2/3})}$; $b \gets \floor{n/a}$; $X \gets 0$; $Y \gets 0$; $Z \gets 0$; $m \gets 0$; $s \gets \isqrt{n}$; $d \gets b$

\lIf{$\isqrt{n} = \floor{n/\isqrt{n}}$}{$s \gets s-1$}

$\chi \gets \floor{n/s}$; $\gamma \gets \isqrt{n/d}$

Prepare a segmented sieve to compute $\mu(x)$ for $1 \leq x \leq a$.

Let $M$ be an array indexed from $1$ through $\isqrt{n}$, inclusive. \label{7-7}

Let $M^\prime$ be an array indexed from $1$ to $\floor{n/\isqrt{n}}$, inclusive, initialized to all zeros.
\begin{multicols}{2}
\For{$x=1$ \KwTo $a$}{
    $v \gets \floor{n/x}$
    
    $m \gets m + \mu(x)$
    
    $X \gets X + \mu(x) \cdot \dfrac{v \cdot (v+1)}{2}$
    
    \uIf{$x \leq \isqrt{n}$}{
        $M_x \gets m$ \label{7-14}
        
        \addtocounter{AlgoLine}{1}
        
        \If{$x > 1$}{ \label{7-15}
            \For{$y=1$ \KwTo $\min(b,\floor{v/x})$}{
                $M^\prime_y \gets M^\prime_y - \mu(x) \cdot \dfloordiv{v}{y}$ \label{7-17}
            }
        }
        \While{$x = \gamma$}{ \label{7-18}
            $M^\prime_d \gets M^\prime_d + 1 - \floor{n/d} + mx$
            
            $d \gets d - 1$
            
            $\gamma \gets \isqrt{n/d}$ \label{7-21}
        }
    }
    
    \addtocounter{AlgoLine}{8}
    
    \ElseIf{$x = \chi$}{
    
        \lIf{$v \neq b$}{$M^\prime_{v} \gets m$}
        
        \addtocounter{AlgoLine}{2}
    
        $s \gets s-1$
        
        $\chi \gets \floor{n/s}$
    }
    
    \addtocounter{AlgoLine}{5}
    
    \lIf{$x = a$}{$Z \gets m \cdot \dfrac{b \cdot (b+1)}{2}$}
}

\emph{lines \ref{7-30}--\ref{7-39} here}

\columnbreak

\vspace*{\fill}

\For{$y=b$ \KwTo $1$}{ \label{7-30}
    $v \gets \floor{n/y}$
    
    $m \gets 0$ \label{7-31}
    
    \For{$x=2$ \KwTo $\isqrt{v}$}{
        \addtocounter{AlgoLine}{1}
        
        \uIf{$\floor{v/x} \leq \isqrt{n}$}{ \label{7-33}
            $m \gets m - M_{\floor{v/x}}$ \label{7-34}
        }
        \Else{ \label{7-35}
            $m \gets m - M^\prime_{\floor{n/\floor{v/x}}}$
        }
    }
    $M^\prime_y \gets M^\prime_y + m$
    
    $Y \gets Y + y \cdot M^\prime_y$ \label{7-39}
}

\end{multicols}

\KwRet $X + Y - Z$
\end{algorithm}

The next step is to move line \ref{7-34} into phase 1.  The work this line does is essentially

\begin{algorithm}[H] \label{Algo7_extract}
\DontPrintSemicolon
\caption{An extract from Algorithm \ref{Algo7}}
\For{$y=b$ \KwTo $1$}{
    $v \gets \floor{n/y}$
    
    $m \gets 0$
    
    \For{$x=2$ \KwTo $\isqrt{v}$}{
        \If{$\floor{v/x} \leq \isqrt{n}$}{
            $m \gets m - M_{\floor{v/x}}$
        }
    }
    $M^\prime_y \gets M^\prime_y + m$
}
\end{algorithm}

This extract iterates over all pairs $(x,y)$ with
\neqn{1 \leq y \leq b \qquadtext{and} 2 \leq x \leq \isqrt{n/y} \qquadtext{and} \floordiv{n}{xy} \leq \isqrt{n} \label{bafnehkj}}
and, for each such pair, subtracts $M_{\floor{n/(xy)}}$ from $M^\prime_y$.  Note that the third inequality is equivalent to this action all happening during phase 1.  Let $k = \floor{n/(xy)}$.  Then for each Mertens value $M(k)$ that we compute, we must find all pairs of integers $(x,y)$ subject to the above bounds and
\eqn{k \leq \frac{n}{xy} < k+1,}
or equivalently,
\eqn{\frac{n}{k+1} < xy \leq \frac{n}{k}.}
Handling a single $k$ at a time is awfully close to factoring $\floor{n/k}$.  To avoid breaking the clock, we will instead gather a block of consecutive Mertens values and handle them all at once.  When this algorithm is fully developed, the memory usage will be $\softO(\sqrt[3]{n})$ due to the array $M^\prime$ and storage inside the M\"{o}bius siever; we will therefore gather Mertens batches of size $b$.  The high index of each batch will be $x$, and the low index will be $A \defeq 1 + b \cdot \floor{x/b}$.  The result is that, when processing each batch, we will be looking for all pairs $(t, \ell)$ such that
\eqn{1 \leq t \leq b \qquadtext{and} 2 \leq \ell \leq \sqrt{n/t} \qquadtext{and} A = b \cdot \floordiv{x}{b} + 1 \qquadtext{and} A \leq \floordiv{n}{\ell t} \leq x.}
Since $A$ and $x$ are integers, the rightmost condition is equivalent to
\eqn{A \leq \frac{n}{\ell t} < x + 1,}
or equivalently,
\eqn{\frac{n}{t \cdot (x+1)} < \ell \leq \frac{n}{At}.}
Furthermore, since $t \leq b$ and the relevant $x$-values are $\leq \isqrt{n}$, the lesser side of this inequality is at least $\softTheta(n^{1/6})$; therefore, the restriction $2 \leq \ell$ above is superfluous.

Algorithm \ref{Algo7_extract} is therefore essentially equivalent to

\begin{algorithm}[H] \label{Algo7_extract_redone}
\DontPrintSemicolon
\caption{Algorithm \ref{Algo7_extract}, redone}
\For{$x=1$ \KwTo $a$}{
    \If{$x \leq \sqrt{n}$}{
        $\mathcal{M}_x \gets M(x)$
        
        \If{$b \mid x$, or $x = \isqrt{n}$,}{
            Let $A$ be the least index in $\mathcal{M}$.
            
            \For{$t=1$ \KwTo $b$}{
                $\ell_{min} \gets 1 + \dfloordiv{n}{t \cdot (x+1)}$
                
                $\ell_{max} \gets \min\left( \isqrt{n/t} , \dfloordiv{n}{t \cdot A} \right)$
                
                \For{$\ell=\ell_{min}$ \KwTo $\ell_{max}$}{
                    $M^\prime_t \gets M^\prime_t - \mathcal{M}_{\floor{n/(\ell t)}}$
                }
            }
            
            Forget the contents of $\mathcal{M}$.
        }
    }
}
\end{algorithm}

Applying this edit to Algorithm \ref{Algo7} yields Algorithm \ref{Algo10}:
\begin{itemize}
\item Line \ref{7-7} has been modified.
\item Lines \ref{10-22}--\ref{10-29} have been inserted.
\item Lines \ref{7-33} and \ref{7-34} have been deleted.
\item Line \ref{7-35} has been modified.
\end{itemize}

\begin{algorithm}[H] \label{Algo10}
\DontPrintSemicolon 
\caption{Compute $\Phi(n)$ in $\softTheta(n^{2/3})$ time and $\softTheta(n^{1/2})$ space.}
\KwData{$n \geq 1$}
\KwResult{$\Phi(n)$}
$a \gets \floor{\softTheta(n^{2/3})}$; $b \gets \floor{n/a}$; $X \gets 0$; $Y \gets 0$; $Z \gets 0$; $m \gets 0$; $s \gets \isqrt{n}$; $d \gets b$

\lIf{$\isqrt{n} = \floor{n/\isqrt{n}}$}{$s \gets s-1$}

$\chi \gets \floor{n/s}$; $\gamma \gets \isqrt{n/d}$

Prepare a segmented sieve to compute $\mu(x)$ for $1 \leq x \leq a$.

Let $\mathcal{M}$ be an array of size $b$.  Its indexing will vary as the algorithm executes. \label{10-5}

Let $M^\prime$ be an array indexed from $1$ to $\floor{n/\isqrt{n}}$, inclusive, initialized to all zeros. \label{10-6}

\begin{multicols}{2}
\For{$x=1$ \KwTo $a$}{
    $v \gets \floor{n/x}$
    
    $m \gets m + \mu(x)$
    
    $X \gets X + \mu(x) \cdot v \cdot (v+1) / 2$
    
    \uIf{$x \leq \isqrt{n}$}{
        $\mathcal{M}_x \gets m$ \label{10-14}
        
        \addtocounter{AlgoLine}{1}
        
        \If{$x > 1$}{
            \For{$y=1$ \KwTo $\min(b,\floor{v/x})$}{
                $M^\prime_y \gets M^\prime_y - \mu(x) \cdot \floor{v/y}$
            }
        }
        \While{$x = \gamma$}{
            $M^\prime_d \gets M^\prime_d + 1 - \floor{n/d} + mx$
            
            $d \gets d - 1$
            
            $\gamma \gets \isqrt{n/d}$
        }
        \If{$b \mid x$, or $x = \isqrt{n}$,}{ \label{10-22}
            Let $A$ be the least index in $\mathcal{M}$.
            
            \For{$t=1$ \KwTo $b$}{
                $\ell_{min} \gets 1 + \floor{n / (t \cdot (x+1))}$
                
                $\ell_{max} \gets \min\left( \isqrt{n/t} , \dfloordiv{n}{t \cdot A} \right)$
                
                \For{$\ell=\ell_{min}$ \KwTo $\ell_{max}$}{
                    $M^\prime_t \gets M^\prime_t - \mathcal{M}_{\floor{n/(\ell t)}}$
                }
            }
            
            Forget the contents of $\mathcal{M}$. \label{10-29}
        }
    }
    \ElseIf{$x = \chi$}{
    
        \lIf{$v \neq b$}{$M^\prime_{v} \gets m$} \label{10-30}
        
        \addtocounter{AlgoLine}{2}
        
        $s \gets s-1$
        
        $\chi \gets \floor{n/s}$
    }
    
    \addtocounter{AlgoLine}{5}
    
    \lIf{$x = a$}{$Z \gets m \cdot b \cdot (b+1) / 2$}
}

\emph{lines \ref{10-38}--\ref{10-45} here}

\columnbreak

\vspace*{\fill}

\For{$y=b$ \KwTo $1$}{ \label{10-38}
    $v \gets \floor{n/y}$
    
    $m \gets 0$
    
    \For{$t=2$ \KwTo $\isqrt{v}$}{
        \addtocounter{AlgoLine}{3}
        
        \If{$\dfloordiv{v}{t} > \isqrt{n}$}{ \label{10-41}
            $m \gets m - M^\prime_{\floor{n/\floor{v/t}}}$ \label{10-39}
        }
    }
    $M^\prime_y \gets M^\prime_y + m$
    
    $Y \gets Y + y \cdot M^\prime_y$ \label{10-45}
}

\end{multicols}

\KwRet $X + Y - Z$
\end{algorithm}

The only thing left to do is remove the need to store $M^\prime_k$ for $k > b$.  This entails moving the action of line \ref{10-39} into phase 2.  The work done by this line is essentially

\begin{algorithm}[H] \label{Algo10_extract}
\DontPrintSemicolon
\caption{An extract from Algorithm \ref{Algo10}}
\For{$y=b$ \KwTo $1$}{
    $v \gets \floor{n/y}$
    
    $m \gets 0$
    
    \For{$t=2$ \KwTo $\isqrt{v}$}{
        \If{$\floor{v/t} > \isqrt{n}$}{
            $m \gets m - M^\prime_{\floor{n/\floor{v/t}}}$
        }
    }
    $M^\prime_y \gets M^\prime_y + m$
}
\end{algorithm}

Let $x$ be the index of a Mertens value that gets saved during phase 2 (so that $\isqrt{n} < x \leq a$).  It gets stored as $M^\prime_\floor{n/x}$, and the $t$- and $y$-values of this loop that touch it are exactly those that satisfy
\eqn{1 \leq y \leq b \quadtext{and} 2 \leq t \leq \isqrt{n/y} \quadtext{and} \isqrt{n} < \floordiv{n}{ty} \quadtext{and} \floordiv{n}{x} = \floordiv{n}{\floor{n/(ty)}}.}
As with the transition from Algorithm \ref{Algo7} to Algorithm \ref{Algo10}, we need to assemble a batch of Mertens values and process them all at once to avoid breaking the clock.  To accommodate this, for each $x$-value such that $M(x)$ gets assembled into the batch, let $w=\floor{n/x}$.  Let $A$ be the greatest $w$-value in the batch and let $B$ be the least.  Then the restrictions on $t$ and $y$ become
\neqn{1 \leq y \leq b \quadtext{and} 2 \leq t \leq \isqrt{n/y} \quadtext{and} \isqrt{n} < \floordiv{n}{ty} \quadtext{and} B \leq \floordiv{n}{\floor{n/(ty)}} \leq A. \label{cbafnlj}}
Applying this edit to Algorithm \ref{Algo10} yields Algorithm \ref{Algo12}:
\begin{itemize}
\item Line \ref{10-6} has been modified.
\item Line \ref{10-30} has been modified.
\item Lines \ref{12-31} and \ref{12-32} have been inserted.
\item Lines \ref{12-35}--\ref{12-39} have been inserted.
\item Line \ref{10-41} has been modified.
\end{itemize}

\begin{algorithm}[H] \label{Algo12}
\DontPrintSemicolon \footnotesize
\caption{Compute $\Phi(n)$ in $\softTheta(n^{2/3})$ time and $\softTheta(n^{1/3})$ space.}
\KwData{$n \geq 1$}
\KwResult{$\Phi(n)$}
$a \gets \floor{\softTheta(n^{2/3})}$; $b \gets \floor{n/a}$; $X \gets 0$; $Y \gets 0$; $Z \gets 0$; $m \gets 0$; $s \gets \isqrt{n}$; $d \gets b$

\lIf{$\isqrt{n} = \floor{n/\isqrt{n}}$}{$s \gets s-1$}

$\chi \gets \floor{n/s}$; $\gamma \gets \isqrt{n/d}$

Prepare a segmented sieve to compute $\mu(x)$ for $1 \leq x \leq a$.

Let $\mathcal{M}$ be an array of size $b$.  Its indexing will vary as the algorithm executes.

Let $M^\prime$ be an array indexed from $1$ to $b$, inclusive, initialized to all zeros.

\begin{multicols}{2}
\For{$x=1$ \KwTo $a$}{
    $v \gets \floor{n/x}$
    
    $m \gets m + \mu(x)$
    
    $X \gets X + \mu(x) \cdot v \cdot (v+1) / 2$
    
    \uIf{$x \leq \isqrt{n}$}{
        $\mathcal{M}_x \gets m$
        
        \addtocounter{AlgoLine}{1}
        
        \If{$x > 1$}{ \label{12-14}
            \For{$y=1$ \KwTo $\min(b,\floor{v/x})$}{
                $M^\prime_y \gets M^\prime_y - \mu(x) \cdot \floor{v/y}$ \label{12-16}
            }
        }
        \While{$x = \gamma$}{
            $M^\prime_d \gets M^\prime_d + 1 - \floor{n/d} + mx$ \label{12-18}
            
            $d \gets d - 1$
            
            $\gamma \gets \isqrt{n/d}$ \label{12-20}
        }
        \If{$b \mid x$, or $x = \isqrt{n}$,}{ \label{12-21}
            Let $A$ be the least index in $\mathcal{M}$.
            
            \For{$t=1$ \KwTo $b$}{
                $\ell_{min} \gets 1 + \floor{n / (t \cdot (x+1))}$
                
                $\ell_{max} \gets \min\left( \isqrt{n/t} , \dfloordiv{n}{t \cdot A} \right)$
                
                \For{$\ell=\ell_{min}$ \KwTo $\ell_{max}$}{
                    $M^\prime_t \gets M^\prime_t - \mathcal{M}_{\floor{n/(\ell t)}}$ \label{12-27}
                }
            }
            
            Forget the contents of $\mathcal{M}$.
        }
    }
    \ElseIf{$x = \chi$}{
    
        \If{$v \neq b$}{
            
            \lIf{$\mathcal{M}$ is empty}{$A \gets v$} \label{12-31}
            
            $\mathcal{M}_v \gets m$; $B \gets v$ \label{12-32}
        }
        
        $s \gets s-1$
        
        $\chi \gets \floor{n/s}$
    }
    \If{$x = a$ or ($x > \isqrt{n}$ and $\mathcal{M}$ is full)}{ \label{12-35}
        \For{$y=1$ to $b$}{
            \ForAll{$t$ satisfying (\ref{cbafnlj})}{
                $M^\prime_y \gets M^\prime_y - \mathcal{M}_{ty}$
            }
        }
        Forget the contents of $\mathcal{M}$. \label{12-39}
    }
    \lIf{$x = a$}{$Z \gets m \cdot b \cdot (b+1) / 2$}
}

\emph{lines \ref{12-42}--\ref{12-49} here}

\columnbreak

\vspace*{\fill}

\For{$y=b$ \KwTo $1$}{ \label{12-42}
    $v \gets \floor{n/y}$
    
    $m \gets 0$
    
    \For{$t=2$ \KwTo $\isqrt{v}$}{ \label{12-45}
        \addtocounter{AlgoLine}{3}
        
        \If{$\dfloordiv{v}{t} > \isqrt{n}$ and $\dfloordiv{n}{\floor{v/t}} \leq b$}{ \label{12-46}
            $m \gets m - M^\prime_{\floor{n/\floor{v/t}}}$ \label{12-47}
        }
    }
    $M^\prime_y \gets M^\prime_y + m$
    
    $Y \gets Y + y \cdot M^\prime_y$ \label{12-49}
}

\end{multicols}

\KwRet $X + Y - Z$
\end{algorithm}

We have now hit our target time- and space-complexities, but some further optimization can be done.  In particular, the iteration over $t$ in phase 3---lines \ref{12-45}--\ref{12-47}---can be made more efficient.  The conditions on $t$ are
\eqn{2 \leq t \leq \isqrt{n/y} \qquadtext{and} \isqrt{n} < \floordiv{\floor{n/y}}{t} \qquadtext{and} \floordiv{n}{\floor{\floor{n/y}/t}} \leq b,}
which is equivalent to
\eqn{2 \leq t \leq \isqrt{n/y} \qquadtext{and} \isqrt{n}+1 \leq \frac{\floor{n/y}}{t} \qquadtext{and} \frac{n}{\floor{\floor{n/y}/t}} < b+1.}
Since $t$ is an integer, we can drop the floor function from the first inequality.
\eqn{2 \leq t \leq \sqrt{n/y} \qquadtext{and} t \leq \frac{\floor{n/y}}{\isqrt{n}+1} \qquadtext{and} \frac{n}{b+1} < \floordiv{\floor{n/y}}{t}.}
By Lemmas \ref{lemma1} and \ref{pmoiuvwq}, we can drop the $t \leq \sqrt{n/y}$ condition and the middle inequality, leaving us with
\eqn{2 \leq t \qquadtext{and} \frac{n}{b+1} < \floordiv{\floor{n/y}}{t}.}

A further improvement can be had by inserting ``and $\mu(x) \neq 0$" into line \ref{12-14}.

Making the corresponding edits modifies lines \ref{12-14}, \ref{12-45}, and \ref{12-46} to yield Algorithm \ref{Algo13}.

\begin{algorithm}[H] \label{Algo13}
\DontPrintSemicolon \footnotesize
\caption{Compute $\Phi(n)$ in $\softTheta(n^{2/3})$ time and $\softTheta(n^{1/3})$ space.}
\KwData{$n \geq 1$}
\KwResult{$\Phi(n)$}
$a \gets \floor{\softTheta(n^{2/3})}$; $b \gets \floor{n/a}$; $X \gets 0$; $Y \gets 0$; $Z \gets 0$; $m \gets 0$; $s \gets \isqrt{n}$; $d \gets b$

\lIf{$\isqrt{n} = \floor{n/\isqrt{n}}$}{$s \gets s-1$}

$\chi \gets \floor{n/s}$; $\gamma \gets \isqrt{n/d}$

Prepare a segmented sieve to compute $\mu(x)$ for $1 \leq x \leq a$.

Let $\mathcal{M}$ be an array of size $b$.  Its indexing will vary as the algorithm executes.

Let $M^\prime$ be an array indexed from $1$ to $b$, inclusive, initialized to all zeros.

\begin{multicols}{2}
\For{$x=1$ \KwTo $a$}{
    $v \gets \floor{n/x}$
    
    $m \gets m + \mu(x)$
    
    $X \gets X + \mu(x) \cdot v \cdot (v+1) / 2$
    
    \uIf{$x \leq \isqrt{n}$}{
        $\mathcal{M}_x \gets m$
        
        \addtocounter{AlgoLine}{1}
        
        \If{$x > 1$ and $\mu(x) \neq 0$}{
            \For{$y=1$ \KwTo $\min(b,\floor{v/x})$}{
                $M^\prime_y \gets M^\prime_y - \mu(x) \cdot \floor{v/y}$ \label{13-16}
            }
        }
        \While{$x = \gamma$}{
            $M^\prime_d \gets M^\prime_d + 1 - \floor{n/d} + mx$ \label{13-18}
            
            $d \gets d - 1$
            
            $\gamma \gets \isqrt{n/d}$ \label{13-20}
        }
        \If{$b \mid x$, or $x = \isqrt{n}$,}{ \label{13-21}
            Let $A$ be the least index in $\mathcal{M}$.
            
            \For{$t=1$ \KwTo $b$}{
                $\ell_{min} \gets 1 + \floor{n / (t \cdot (x+1))}$
                
                $\ell_{max} \gets \min\left( \isqrt{n/t} , \dfloordiv{n}{t \cdot A} \right)$
                
                \For{$\ell=\ell_{min}$ \KwTo $\ell_{max}$}{
                    $M^\prime_t \gets M^\prime_t - \mathcal{M}_{\floor{n/(\ell t)}}$ \label{13-27}
                }
            }
            
            Forget the contents of $\mathcal{M}$.
        }
    }
    \ElseIf{$x = \chi$}{
    
        \If{$v \neq b$}{
            
            \lIf{$\mathcal{M}$ is empty}{$A \gets v$}
            
            $\mathcal{M}_v \gets m$; $B \gets v$
        }
        
        $s \gets s-1$
        
        $\chi \gets \floor{n/s}$
    }
    \If{$x = a$ or ($x > \isqrt{n}$ and $\mathcal{M}$ is full)}{
        \For{$y=1$ to $b$}{
            \ForAll{$t$ satisfying (\ref{cbafnlj})}{
                $M^\prime_y \gets M^\prime_y - \mathcal{M}_{ty}$ \label{13-38}
            }
        }
        Forget the contents of $\mathcal{M}$.
    }
    \lIf{$x = a$}{$Z \gets m \cdot b \cdot (b+1) / 2$}
}

\emph{lines \ref{13-42}--\ref{13-50} here}

\columnbreak

\vspace*{\fill}

\For{$y=b$ \KwTo $1$}{ \label{13-42}
    $v \gets \floor{n/y}$
    
    $m \gets 0$
    
    \For{$t \in \set{2,3,4,...}$}{ \label{13-45}
        \addtocounter{AlgoLine}{3}
        
        \lIf{$\disp \frac{n}{b+1} \geq \floordiv{v}{t}$}{\KwBreak} \label{13-46}
        
        $m \gets m - M^\prime_{\floor{n/\floor{v/t}}}$ \label{13-48}
    }
    $M^\prime_y \gets M^\prime_y + m$
    
    $Y \gets Y + y \cdot M^\prime_y$ \label{13-50}
}

\end{multicols}

\KwRet $X + Y - Z$
\end{algorithm}

\section{Analysis of Algorithm \ref{Algo13}} \label{Analysis}

The inner loops of Algorithm \ref{Algo13} are inside the M\"{o}bius siever and lines \ref{13-16}, \ref{13-18}--\ref{13-20}, \ref{13-27}, \ref{13-38}, and \ref{13-46}--\ref{13-48}.

\begin{lemma} \label{13-M-time}
The M\"{o}bius sieving consumes $\Theta(a \ln(\ln(a)))$ time.
\end{lemma}
\begin{proof}
The M\"{o}bius function is sieved up to $a$; the rest is a standard result.
\end{proof}

\begin{lemma} \label{13-16-time}
Line \ref{13-16} consumes $\Theta(n/\sqrt{a})$ time.
\end{lemma}
\begin{proof}
Line \ref{13-16} is hit
\eqn{\sum_{\substack{2 \leq x \leq \isqrt{n} \\ \mu(x) \neq 0}} \min\left(\floordiv{n}{a}, \floordiv{n}{x^2}\right)}
times.  Since the squarefree integers have a natural density of $6/\pi^2$, this is
\eqn{= \Theta\left( \sum_{x=2}^{\isqrt{n}} \min\left(\floordiv{n}{a}, \floordiv{n}{x^2}\right) \right)}
\eqn{= \Theta\left( \sum_{x=2}^{\isqrt{a}} \floordiv{n}{a} + \sum_{x=\isqrt{a}+1}^{\isqrt{n}} \floordiv{n}{x^2} \right)}
\eqn{= \Theta\left( \floordiv{n}{a} \cdot \left(\isqrt{a} - 1\right) + \integral{\;\frac{n}{x^2}}{x}{\isqrt{a}+1}{\isqrt{n}}\right)}
\eqn{= \Theta\left(\floordiv{n}{a} \cdot \left(\isqrt{a} - 1\right) + \frac{n}{\isqrt{a}+1} - \frac{n}{\isqrt{n}}\right)}
\eqn{= \Theta\left(\frac{n}{\sqrt{a}}\right),}
as desired.
\end{proof}

\begin{lemma} \label{Algo13mintime}
The runtime of Algorithm \ref{Algo13} is at least $\Theta(n^{2/3} \cdot (\ln(\ln(n)))^{1/3})$.
\end{lemma}
\begin{proof}
By Lemmas \ref{13-M-time} and \ref{13-16-time}, the combined runtime of of the M\"{o}bius sieving and line \ref{13-16} is
\eqn{\Theta(a \ln(\ln(a))) + \Theta\left(\frac{n}{\sqrt{a}}\right),}
which is minimized by choosing
\eqn{a = \Theta\left(\left(\frac{n}{\ln(\ln(n))}\right)^{2/3}\right);}
the runtime of those parts is then $\Theta\left(n^{2/3} \cdot (\ln(\ln(n)))^{1/3}\right)$ each.
\end{proof}

Lines \ref{13-18}--\ref{13-20} can be neglected: they get hit at most $b$ times.

\begin{lemma} \label{13-27-time}
With $a = \Theta((n/\ln(\ln(n)))^{2/3})$, line \ref{13-27} consumes $\Theta\left(n^{2/3} \cdot (\ln(\ln(n)))^{1/3}\right)$ time.
\end{lemma}
\begin{proof}
Line \ref{13-27} gets hit at least
\eqn{\sum_{\substack{1 \leq x \leq \sqrt{n} \\ b \mid x}}\sum_{t=1}^b \left( 1 + \max\left( \ell_{max}(x,t) - \ell_{min}(x,t), 0 \right) \right)}
times.  If $b \mid \isqrt{n}$, then this is exactly the number of hits; otherwise, there will be a final phase-1 batch that is not included in the above sum.

The first batch ($x=b$) and that possible last batch $(x=\isqrt{n} \not\equiv 0 \pmod{b})$ need special handling.  For the first batch, when $x=b$, the relevant section of the algorithm amounts to

\begin{algorithm}[H]
\DontPrintSemicolon
\caption{An extract from Algorithm \ref{Algo13}, with $x=b$}
\Begin{
    \For{$t=1$ \KwTo $b$}{
        $\ell_{min} \gets 1 + \dfloordiv{n}{t \cdot (b+1)}$
        
        $\ell_{max} \gets \min\left( \isqrt{n/t} , \dfloordiv{n}{t} \right)$
        
        \For{$\ell=\ell_{min}$ \KwTo $\ell_{max}$}{
            $M^\prime_t \gets M^\prime_t - \mathcal{M}_{\floor{n/(\ell t)}}$
        }
    }
}
\end{algorithm}

The time consumed by this extract is
\eqn{\sum_{t=1}^b \left( 1 + \max\left( \ell_{max} - \ell_{min} , 0 \right) \right)}
\eqn{= \sum_{t=1}^b \left( 1 + \max\left( \min\left( \isqrt{n/t} , \floordiv{n}{t} \right) - 1 - \floordiv{n}{t \cdot (b+1)} , 0 \right) \right)}
\eqn{= b + \sum_{t=1}^b \max\left( \min\left( \isqrt{n/t} , \floordiv{n}{t} \right) - 1 - \floordiv{n}{t \cdot (b+1)} , 0 \right).}
We have $0 < t < n$, so $\sqrt{n/t} \leq n/t$.
\eqn{= b + \sum_{t=1}^b \max\left( \isqrt{n/t} - 1 - \floordiv{n}{t \cdot (b+1)} , 0 \right)}
For $b \leq n^{1/3}$, Lemma \ref{ouawrt4coi} establishes that the first argument of the $\max$ function is negative, so the time consumed by the extract is $O(b)$.  The case $b > n^{1/3}$ is a bit more complicated: in this case, the $\max$ function's first argument prevails for
\eqn{t > t_0 \defeq n \cdot \left(\frac{1}{2} - \frac{1}{b+1} - \sqrt{\frac{1}{4} - \frac{1}{b+1}}\right),}
so the time consumed by the extract is
\eqn{b + \sum_{t=t_0}^b \left( \isqrt{n/t} - 1 - \floordiv{n}{t \cdot (b+1)} \right).}
The sum will turn out to be $\softTheta(a)$, so the $b$ is neglectably small.
\eqn{= \Theta\left(\int_{t_0}^b \left( \sqrt{\frac{n}{t}} - 1 - \frac{n}{t \cdot (b+1)} \right) dt\right)}
\eqn{= \Theta\left(2\sqrt{bn} - 2\sqrt{nt_0} - b + t_0 + \frac{n}{b+1} \cdot \ln\left(\frac{t_0}{b}\right)\right)}
From Lemma \ref{sqrtapprox2}, we have that $t_0 = n \cdot (b^{-2} + O(b^{-3}))$.  Therefore,
\eqn{= \Theta\left(2\sqrt{bn} - 2\sqrt{n^2 \cdot (b^{-2} + O(b^{-3}))} - b + n \cdot (b^{-2} + O(b^{-3})) + \frac{n}{b+1} \cdot \ln\left(\frac{n \cdot (b^{-2} + O(b^{-3}))}{b}\right)\right)}
\eqn{= \Theta\left(2\sqrt{bn} - 2a\sqrt{1 + O(b^{-1})} - b + \frac{n}{b^2} \cdot (1+O(b^{-1})) + \frac{a}{1 + 1/b} \cdot \ln\left(\frac{n}{b^3}\cdot(1+O(b^{-1}))\right)\right).}
Recall that we are working with $a = \Theta\left((n/\ln(\ln(n)))^{2/3}\right)$, so $b = \Theta\left(n^{1/3} \cdot (\ln(\ln(n)))^{2/3}\right)$.  Therefore the first term dominates the rest, yielding
\eqn{= \Theta\left(n^{2/3} \cdot (\ln(\ln(n)))^{1/3}\right).}
In the last batch, when $x = \isqrt{n} \not\equiv 0 \pmod{b}$, the relevant section of the algorithm amounts to

\begin{algorithm}[H]
\DontPrintSemicolon
\caption{An extract from Algorithm \ref{Algo13}, with $x = \isqrt{n} \not\equiv 0 \pmod{b}$}
\Begin{
    $x \gets \isqrt{n}$
    
    $A \gets 1 + b \cdot \floor{x/b}$
    
    \For{$t=1$ \KwTo $b$}{
        $\ell_{min} \gets 1 + \dfloordiv{n}{t \cdot (x+1)}$
        
        $\ell_{max} \gets \min\left( \isqrt{n/t} , \dfloordiv{n}{t \cdot A} \right)$
        
        \For{$\ell=\ell_{min}$ \KwTo $\ell_{max}$}{
            $M^\prime_t \gets M^\prime_t - \mathcal{M}_{\floor{n/(\ell t)}}$
        }
    }
}
\end{algorithm}

The time consumed by this extract is
\eqn{\sum_{t=1}^b \left( 1 + \max\left( \ell_{max} - \ell_{min} , 0 \right) \right)}
\eqn{= b + \sum_{t=1}^b \max\left( \min\left( \isqrt{n/t} , \dfloordiv{n}{t \cdot A} \right) - 1 - \dfloordiv{n}{t \cdot (x+1)} , 0 \right)}
\neqn{= b + \sum_{t=1}^b \max\left( \min\left( \isqrt{n/t} , \dfloordiv{n}{t \cdot \left(1 + b \cdot \dfloordiv{\isqrt{n}}{b}\right)} \right) - 1 - \dfloordiv{n}{t \cdot (\isqrt{n}+1)} , 0 \right). \label{hydrogen}}
We split off the $t=1$ term for special treatment: this is
\eqn{\max\left( \min\left( \isqrt{n} , \dfloordiv{n}{1 + b \cdot \dfloordiv{\isqrt{n}}{b}} \right) - 1 - \dfloordiv{n}{\isqrt{n}+1} , 0 \right).}
By hypothesis, we are working with $\isqrt{n} \not\equiv 0 \pmod{b}$, so $1 + b \cdot \floor{\isqrt{n}/b} \leq \isqrt{n}$.  Therefore, in the $\min$ function, the first argument prevails.
\eqn{= \max\left( \isqrt{n} - 1 - \dfloordiv{n}{\isqrt{n}+1} , 0 \right)}
\eqn{= 0}
Therefore (\ref{hydrogen}) becomes
\eqn{= b + \sum_{t=2}^b \max\left( \min\left( \isqrt{n/t} , \dfloordiv{n}{t \cdot \left(1 + b \cdot \dfloordiv{\isqrt{n}}{b}\right)} \right) - 1 - \dfloordiv{n}{t \cdot (\isqrt{n}+1)} , 0 \right).}
By Lemma \ref{okiagfvwr}, in the $\min$ function, the second argument prevails.
\eqn{= b + \sum_{t=2}^b \max\left( \floordiv{n}{t \cdot \left(1 + b \cdot \dfloordiv{\isqrt{n}}{b}\right)} - 1 - \floordiv{n}{t \cdot (\isqrt{n}+1)} , 0 \right)}
\eqn{= \Theta\left( b + \int_2^b \max\left( \frac{n}{t \cdot \left(1 + b \cdot \dfloordiv{\isqrt{n}}{b}\right)} - 1 - \frac{n}{t \cdot (\isqrt{n}+1)} , 0 \right) dt \right)}
Let
\eqn{U \defeq n \cdot \frac{\isqrt{n} - b \cdot \dfloordiv{\isqrt{n}}{b}}{\left(1 + b \cdot \dfloordiv{\isqrt{n}}{b}\right)\left(\isqrt{n}+1\right)}.}
Then
\eqn{= \Theta\left( b + \int_2^b \max\left( \frac{U}{t} - 1 , 0 \right) dt \right)}
\eqn{< \Theta\left( b + \int_2^b \max\left( \frac{U}{t} , 0 \right) dt \right)}
\eqn{< \Theta\left( b + U \ln(U) \right).}
By Lemma \ref{lmmfanb}, this can be weakened to
\eqn{= O(b + \sqrt{n} \ln(n)).}

Recall that we are in the process of estimating the time devoted to line \ref{13-27}.  We have separated out the first and last batches for special treatment, and found them to consume
\eqn{\Theta\left(n^{2/3} \cdot (\ln(\ln(n)))^{1/3}\right) + O(b + \sqrt{n} \ln(n))}
\neqn{= \Theta\left(n^{2/3} \cdot (\ln(\ln(n)))^{1/3}\right) \label{13-27-time-1}}
time.  We now turn our attention to those batches in which $b \mid x$ and $x \neq b$.  For those batches, we have $A = 1 + x - b$, so the time devoted to those batches is
\eqn{\sum_{\substack{2b \leq x \leq \sqrt{n} \\ b \mid x}} \; \sum_{t=1}^b \left( 1 + \max\left( \min \left( \isqrt{n/t} , \floordiv{n}{t \cdot (1 + x - b)} \right) - \left( 1 + \floordiv{n}{t \cdot (x+1)} \right) , 0 \right) \right).}
By Lemma \ref{BigIntegral}, this is
\eqn{= \Theta\left( a \cdot \ln(n^2 a^{-3}) \right).}
With $a = \Theta\left((n/\ln(\ln(n)))^{2/3}\right)$, this works out to
\eqn{= \Theta\left( a \cdot \ln(\ln(\ln(n))) \right)}
\neqn{= \Theta\left( n^{2/3} \cdot (\ln(\ln(n)))^{-2/3} \cdot \ln(\ln(\ln(n))) \right). \label{13-27-time-2}}
From (\ref{13-27-time-1}) and (\ref{13-27-time-2}), we see that, with $a = \Theta((n/\ln(\ln(n)))^{2/3})$, the total time consumed in handling line \ref{13-27} is
\eqn{\Theta\left(n^{2/3} \cdot (\ln(\ln(n)))^{1/3}\right) + \Theta\left( n^{2/3} \cdot (\ln(\ln(n)))^{-2/3} \cdot \ln(\ln(\ln(n))) \right)}
\eqn{= \Theta \left( n^{2/3} \cdot (\ln(\ln(n)))^{1/3} \right),}
as desired.
\end{proof}

\begin{lemma} \label{13-38-time}
Line \ref{13-38} consumes $\Theta\left(\sqrt{n} \ln(b)\right)$ time.
\end{lemma}
\begin{proof}
On each iteration through phase 2 in which line \ref{13-38} gets hit, it gets hit
\eqn{\sum_{y=1}^b f(t,y)}
times, where $f(t,y)$ is the number of integers $t$ that satisfy (\ref{cbafnlj}).

We need to determine which values of $A$ and $B$ happen.  The $k$\textsuperscript{th} value of $x$ such that $M(x)$ gets saved during phase 2 is approximately
\eqn{\frac{n}{\sqrt{n} - k}.}
We accumulate batches of size $b$, so a batch will be processed when $k \mid b$; the first batch has $k=b$ and the last batch will have $k \approx \sqrt{n} - \sqrt[3]{n}$.  The $j$\textsuperscript{th} batch will then have
\eqn{A \approx \sqrt{n} - b \cdot (j-1) \qquadtext{and} B \approx \sqrt{n} - bj.}
When processing the $j$\textsuperscript{th} batch, the bounds on $t$ are therefore approximately
\eqn{2 \leq t \leq \sqrt{n/y} \quadtext{and} \sqrt{n} < \frac{n}{ty} \quadtext{and} \sqrt{n}-bj \leq ty \leq \sqrt{n}-b\cdot(j-1)}
\eqn{2 \leq t \leq \sqrt{\frac{n}{y}} \quadtext{and} t < \frac{\sqrt{n}}{y} \quadtext{and} \frac{\sqrt{n}-bj}{y} \leq t \leq \frac{\sqrt{n}-bj}{y}+\frac{b}{y}.}
The second condition and the upper side of the first condition are both weaker than the upper side of the third condition.
\eqn{2 \leq t \quadtext{and} \frac{\sqrt{n}-bj}{y} \leq t \leq \frac{\sqrt{n}-bj}{y}+\frac{b}{y}}
Therefore
\eqn{f(t,y) \approx \frac{\sqrt{n}-bj}{y}+\frac{b}{y} - \max\left(2, \frac{\sqrt{n}-bj}{y}\right),}
so on an iteration through phase 2 in which line \ref{13-38} gets hit, it gets hit
\eqn{\Theta\left(\sum_{y=1}^b \left( \frac{\sqrt{n}-bj}{y}+\frac{b}{y} - \max\left(2, \frac{\sqrt{n}-bj}{y}\right)\right)\right) = \Theta\left( \sum_{y=1}^b \frac{b}{y} \right) = \Theta\left( b \ln(b) \right)}
times.
There are $\Theta(\sqrt{n}-\sqrt[3]{n})$ values of $x$ such that $M(x)$ gets saved during phase 2; since each batch has size $b$, there are
\eqn{\Theta\left( \frac{\sqrt{n}-\sqrt[3]{n}}{b} \right)}
batches.  The total time devoted to line \ref{13-38} across all iterations through phase 2 is therefore
\eqn{\Theta\left( b \ln(b) \right) \cdot \Theta\left( \frac{\sqrt{n}-\sqrt[3]{n}}{b} \right)}
\eqn{= \Theta\left( \sqrt{n} \ln(b) \right),}
as desired.
\end{proof}

\begin{lemma} \label{13-46-48-time}
Lines \ref{13-46}--\ref{13-48} consume $\Theta(b \ln(b))$ time.
\end{lemma}
\begin{proof}
These lines get hit
\eqn{\Theta\left( \sum_{y=1}^b \frac{n/y}{n/(b+1)} \right) = \Theta\left( \sum_{y=1}^b \frac{b+1}{y} \right) = \Theta\left( b \ln(b) \right)}
times.
\end{proof}

\begin{theorem} \label{Algo13time}
With an optimal parameter selection of
\eqn{a = \Theta\left(\left( \frac{n}{\ln(\ln(n))} \right)^{2/3} \right),}
Algorithm \ref{Algo13} computes $\Phi(n)$ in $\Theta\left(n^{1/3} \cdot (\ln(\ln(n)))^{2/3}\right)$ space and $\Theta\left( n^{2/3} \cdot (\ln(\ln(n)))^{1/3} \right)$ time.
\end{theorem}
\begin{proof}
The correctness of the algorithm is evident by its development, and the space complexity follows by inspection.

For the time complexity, Lemmas \ref{13-M-time}, \ref{13-16-time}, and \ref{Algo13mintime} prove that the minimum possible time for a subsection of the algorithm is $\Theta\left( n^{2/3} \cdot (\ln(\ln(n)))^{1/3} \right)$ and that this is obtained by selecting $a = \Theta\left((n/\ln(\ln(n)))^{2/3}\right)$, while Lemmas \ref{13-27-time}, \ref{13-38-time}, and \ref{13-46-48-time} establish that, with this parmeter selection, the rest of the algorithm does not exceed this time complexity.
\end{proof}

\section{Computational results} \label{ComputationalResults}

The data in Table \ref{Phitable} was procured by running \texttt{totientsum.py} under PyPy3 on a computer with an AMD Ryzen 9 7950X CPU and 128 GB of RAM.  The times are wall-clock times; the space requirements are maximum resident-set sizes reported by the command \texttt{/usr/bin/time -v}.

\begin{table}[h]
\centering
\begin{tabular}{|c|l|r|r|r|r|} \hline
\multirow{2}{*}{$n$} & \multicolumn{1}{c|}{\multirow{2}{*}{$\Phi(10^n)$ \qquad \seqnum{A064018}($n$)}} & \multicolumn{3}{c|}{Time (s)} & Memory \\\cline{3-5}
   &                                        & Phase 1 & Phase 2 &  Total &    (kb) \\\hline
13 & 30396355092702898919527444             &       5 &       8 &     13 &   95612 \\\hline
14 & 3039635509270144893910357854           &      22 &      33 &     56 &  122480 \\\hline
15 & 303963550927013509478708835152         &      96 &     154 &    250 &  193868 \\\hline
16 & 30396355092701332166351822199504       &     459 &     754 &   1214 &  331492 \\\hline
17 & 3039635509270133156701800820366346     &    2114 &    3803 &   5916 &  712492 \\\hline
18 & 303963550927013314319686824781290348   &   10135 &   18691 &  28826 & 1695468 \\\hline
19 & 30396355092701331435065976498046398788 &  614160 &  162208 & 776388 & 6586924 \\\hline
\end{tabular}
\caption{Computation of some values of $\Phi$.}
\label{Phitable}
\end{table}

The values of $\Phi(10^n)$ for $n \in \set{13, \ldots, 18}$ were already known; the values computed here match those at \url{https://oeis.org/A064018/b064018.txt}.

The computation of $\Phi(10^{19})$ is new.  Based on the algorithm's space- and time-complexities, one would expect the computation of $\Phi(10^{19})$ to take about 160,000 seconds and occupy about 3.6 GB.  The reason for the deviation from this expectation is presumably that $10^{18}$ fits in a single 64-bit machine word while $10^{19}$ does not.  To guard against computer glitches, the computation was run twice; the results matched.

\section{Supporting lemmas} \label{SupportingLemmas}

\begin{lemma} \label{lemma1}
Let $n$ and $y$ be integers such that $1 \leq y \leq \sqrt{n}$.  Then
\eqn{\sqrt{\frac{n}{y}} \geq \frac{\floor{n/y}}{\isqrt{n}+1}.}
\end{lemma}
\begin{proof}
We begin with $\isqrt{n} + 1 \geq \sqrt{n}$.  We can then weaken this to
\eqn{\isqrt{n} + 1 \geq \sqrt{\frac{n}{y}}.}
Multiplying by $\sqrt{n/y}/(\isqrt{n}+1)$ yields
\eqn{\sqrt{\frac{n}{y}} \geq \frac{n/y}{\isqrt{n}+1},}
which we can weaken to
\eqn{\sqrt{\frac{n}{y}} \geq \frac{\floor{n/y}}{\isqrt{n}+1},}
as desired.
\end{proof}

\begin{lemma} \label{pmoiuvwq}
With $b$, $n$, $x$, and $y$ as in lines \ref{12-45}--\ref{12-47},
\eqn{\frac{n}{b+1} < \floordiv{\floor{n/y}}{x} \implies x \leq \frac{\floor{n/y}}{\isqrt{n}+1}.}
\end{lemma}
\begin{proof}
The hypothesis can be weakened by dropping the outer floor function, yielding
\eqn{\frac{n}{b+1} < \frac{\floor{n/y}}{x}}
and therefore
\eqn{x \leq \frac{\floor{n/y}}{n/(b+1)}.}
Since $b = \softTheta(n^{1/3})$, we have $n/(b+1) > \isqrt{n} + 1$, and the conclusion follows immediately.
\end{proof}

\begin{lemma} \label{ouawrt4coi}
Suppose that $1 \leq t \leq b \leq n^{1/3}$.  Then for sufficiently large $b$ and $n$, 
\eqn{\isqrt{n/t} \leq 1 + \floordiv{n}{t \cdot (b+1)}.}
\end{lemma}
\begin{proof}
Consider the function
\eqn{f(z) = \left( 1 + \frac{z}{b+1} \right)^2 - z.}
Observe that $f(b^2)$, $f'(b^2)$, and $f''$ are all positive.  Therefore, if $z \geq b^2$, then $f(z) \geq 0$.

Now consider $f(n/t)$.  From the hypotheses of this lemma, we have $n/t \geq b^2$.  Therefore,
\eqn{0 \leq f(n/t) = \left( 1 + \frac{n/t}{b+1} \right)^2 - \frac{n}{t}}
\eqn{\frac{n}{t} \leq \left( 1 + \frac{n}{t \cdot (b+1)} \right)^2}
\eqn{\sqrt{n/t} \leq 1 + \frac{n}{t \cdot (b+1)}.}
Applying the floor function to both sides then yields the desired result.
\end{proof}

\begin{lemma} \label{sqrtapprox2}
$\disp \lim_{u\rightarrow\infty} u^3 \cdot \left( \frac{1}{2} - \frac{1}{u} - \sqrt{\frac{1}{4}-\frac{1}{u}} - \frac{1}{u^2} \right) = 2$.
\end{lemma}
\begin{proof}
Make the substitution $u=1/x$ to obtain
\eqn{\lim_{x\rightarrow0^+} x^{-3} \cdot \left( \frac{1}{2} - x - \sqrt{\frac{1}{4}-x} - x^2 \right)}
\eqn{ = \lim_{x\rightarrow0^+} \frac{1 - 2x - (1-4x)^{1/2} - 2x^2}{2x^3}.}
L'H\^{o}pital's rule yields
\eqn{ = \lim_{x\rightarrow0^+} \frac{0 - 2 - (1/2)(1-4x)^{-1/2}(-4) - 4x}{6x^2}}
\eqn{ = \lim_{x\rightarrow0^+} \frac{-2 + 2(1-4x)^{-1/2} - 4x}{6x^2}.}
L'H\^{o}pital's rule yields
\eqn{ = \lim_{x\rightarrow0^+} \frac{0 + 2(-1/2)(1-4x)^{-3/2}(-4) - 4}{12x}}
\eqn{ = \lim_{x\rightarrow0^+} \frac{4(1-4x)^{-3/2} - 4}{12x}.}
L'H\^{o}pital's rule yields
\eqn{ = \lim_{x\rightarrow0^+} \frac{4(-3/2)(1-4x)^{-5/2}(-4) - 0}{12}}
\eqn{ = \lim_{x\rightarrow0^+} \frac{24(1-4x)^{-5/2}}{12}}
\eqn{ = 2,}
as desired.
\end{proof}

\begin{lemma} \label{okiagfvwr}
Suppose that $2 \leq t$ and $b = o(\sqrt{n})$.  Then for sufficiently large $n$,
\eqn{\floordiv{n}{t \cdot \left(1 + b \cdot \dfloordiv{\isqrt{n}}{b}\right)} \leq \isqrt{n/t}.}
\end{lemma}
\begin{proof}
We begin with
\eqn{\sqrt{\frac{n}{2}} < \isqrt{n} - b + 1 < \isqrt{n} - (\isqrt{n} \bmod b) + 1 = b \cdot \floordiv{\isqrt{n}}{b} + 1.}
The first inequality is true for sufficiently large $n$ because $b = o(n^{1/2})$; the rest is arithmetic.  By hypothesis, $t \geq 2$, so we can weaken this to
\eqn{\sqrt{\frac{n}{t}} < 1 + b \cdot \floordiv{\isqrt{n}}{b}}
and therefore
\eqn{\frac{n/t}{1 + b \cdot \dfloordiv{\isqrt{n}}{b}} < \sqrt{n/t}.}
Applying the floor function to both sides then yields the desired result.
\end{proof}

\begin{lemma} \label{lmmfanb}
$U \leq 2 \sqrt{n}$.
\end{lemma}
\begin{proof}
Since $b = o(\sqrt{n})$, we have for sufficiently large $n$
\eqn{n \leq 2 \cdot \left(1 + \isqrt{n}\right)^2 - 2b \cdot \left(\isqrt{n}+1\right)}
\eqn{ = 2 \cdot \left(1 + \isqrt{n} - b\right)\left(\isqrt{n}+1\right)}
\eqn{ \leq 2 \cdot \left(1 + \isqrt{n} - (\isqrt{n} \bmod b)\right)\left(\isqrt{n}+1\right)}
\eqn{ = 2 \cdot \left(1 + b \cdot \floordiv{\isqrt{n}}{b}\right)\left(\isqrt{n}+1\right),}
which is equivalent to
\eqn{2 \sqrt{n} \geq n \cdot \frac{\sqrt{n}}{\left(1 + b \cdot \dfloordiv{\isqrt{n}}{b}\right)\left(\isqrt{n}+1\right)}}
\eqn{ \geq n \cdot \frac{\isqrt{n}}{\left(1 + b \cdot \dfloordiv{\isqrt{n}}{b}\right)\left(\isqrt{n}+1\right)}}
\eqn{ \geq n \cdot \frac{\isqrt{n} - b \cdot \dfloordiv{\isqrt{n}}{b}}{\left(1 + b \cdot \dfloordiv{\isqrt{n}}{b}\right)\left(\isqrt{n}+1\right)},}
which is the desired result.
\end{proof}

\begin{lemma} \label{BigIntegral}
\eqn{\sum_{\substack{2b \leq x \leq \sqrt{n} \\ b \mid x}} \; \sum_{t=1}^b \left( 1 + \max\left( \min \left( \isqrt{n/t} , \floordiv{n}{t \cdot (1 + x - b)} \right) - \left( 1 + \floordiv{n}{t \cdot (x+1)} \right) , 0 \right) \right)}
\eqn{= \Theta\left( a \cdot \ln(n^2 a^{-3}) \right).}
\end{lemma}
\begin{proof}
Reindexing the outer sum yields
\eqn{\sum_{\chi=2}^{\sqrt{n}/b} \sum_{t=1}^b \left( 1 + \max\left( \min \left( \isqrt{n/t} , \floordiv{n}{t \cdot (1 + b\chi - b)} \right)
- 1 - \floordiv{n}{t \cdot (b\chi + 1)}
, 0 \right) \right)}
\eqn{= \sqrt{n} - 1 + \sum_{\chi=2}^{\sqrt{n}/b} \sum_{t=1}^b \max\left( \min \left( \isqrt{n/t} , \floordiv{n}{t \cdot (1 + b\chi - b)} \right) - 1 - \floordiv{n}{t \cdot (b\chi + 1)} , 0 \right).}
The sum will turn out to be $\softTheta(a)$, so the $\sqrt{n} - 1$ is neglectably small.
\eqn{= \Theta\left( \int_2^{\sqrt{n}/b} \int_1^b \max\left( \min \left( \sqrt{n/t} , \frac{n}{t \cdot (1 + b\chi - b)} \right) - 1 - \frac{n}{t \cdot (b\chi + 1)} , 0 \right) dt \; d\chi \right)}
Let $u=b\chi+1$.
\eqn{= \Theta\left( \frac{1}{b} \cdot \int_{1+2b}^{1+\sqrt{n}} \int_1^b \max\left( \min \left( \sqrt{n/t} , \frac{n}{t \cdot (u - b)} \right) - 1 - \frac{n}{tu} , 0 \right) dt \; du \right)}
Let $T \defeq n \cdot (u-b)^{-2}$.  This is the crossover point in the $\min$ function.  For lesser $t$, the first argument prevails, and the integral becomes
\eqn{= \Theta\left( \frac{1}{b} \cdot \int_{1+2b}^{1+\sqrt{n}} \left( \int_T^b \max\left( \frac{bn}{tu \cdot (u - b)} - 1 , 0 \right) dt + \int_1^T \max\left( \sqrt{n/t} - 1 - \frac{n}{tu} , 0 \right) dt \right) du \right).}
By Lemma \ref{mfeqljk_u}, in the first inner integral, $\max$ function's first argument prevails throughout the interval of integration.
\eqn{= \Theta\left( \frac{1}{b} \cdot \int_{1+2b}^{1+\sqrt{n}} \left( \int_T^b \left( \frac{bn}{tu \cdot (u - b)} - 1 \right) dt + \int_1^T \max\left( \sqrt{n/t} - 1 - \frac{n}{tu} , 0 \right) dt \right) du \right)}
Let $\displaystyle S \defeq \frac{n}{2} - \frac{n}{u} - \sqrt{\left(\frac{n}{2}\right)^2 - \frac{n^2}{u}}$.  This is the crossover point in the remaining $\max$ function; the first argument prevails for $t > S$.
\eqn{= \Theta\left( \frac{1}{b} \cdot \int_{1+2b}^{1+\sqrt{n}} \left( \int_T^b \left( \frac{bn}{tu \cdot (u - b)} - 1 \right) dt + \int_S^T \left( \sqrt{n/t} - 1 - \frac{n}{tu} \right) dt \right) du \right)}
\eqn{= \Theta\left( \frac{1}{b} \cdot \int_{1+2b}^{1+\sqrt{n}} \left( \frac{bn \ln(b/T)}{u \cdot (u - b)} - (b - T) + 2\sqrt{n}(T^{1/2} - S^{1/2}) - (T - S) - \frac{n \ln(T/S)}{u} \right) du \right)}
Cancelling and substituting out the $T$s, and factoring out an $n$, yields
\eqn{= \Theta\left( \frac{n}{b} \cdot \int_{1+2b}^{1+\sqrt{n}} \left( \frac{b \ln\left(\dfrac{b}{n}(u-b)^2\right)}{(u - b) \cdot u} - \frac{b}{n} + \frac{2}{u-b} - 2\sqrt{\frac{S}{n}} + \frac{S}{n} + \frac{\ln\left(\dfrac{S}{n}(u-b)^2\right)}{u} \right) du \right).}
We now pull the $b/n$ term out, and also substitute $a=n/b$.
\eqn{= \Theta\left( 2b - \sqrt{n} + a \cdot \int_{1+2b}^{1+\sqrt{n}} \left( \frac{b \ln\left(\dfrac{(u-b)^2}{a}\right)}{(u - b) \cdot u} + \frac{2}{u-b} - 2\sqrt{\frac{S}{n}} + \frac{S}{n} + \frac{\ln\left(\dfrac{S}{n}(u-b)^2\right)}{u} \right) du \right)}
The expression will turn out to be $\softTheta(a)$, so the $2b-\sqrt{n}$ is neglectably small.
\eqn{= a \cdot \Theta\left( \int_{1+2b}^{1+\sqrt{n}} \left( \frac{b \ln\left(\dfrac{(u-b)^2}{a}\right)}{(u - b) \cdot u} + \frac{2}{u-b} - 2\sqrt{\frac{S}{n}} + \frac{S}{n} + \frac{\ln\left(\dfrac{S}{n}(u-b)^2\right)}{u} \right) du \right)}
Observe that $\displaystyle \sqrt{\frac{S}{n}} = \frac{S}{n} + \frac{1}{u}$.
\eqn{= a \cdot \Theta\left( \int_{1+2b}^{1+\sqrt{n}} \left( \frac{b \ln\left(\dfrac{(u-b)^2}{a}\right)}{(u - b) \cdot u} + \frac{2}{u-b} - \frac{2}{u} - \frac{S}{n} + \frac{\ln\left(\dfrac{S}{n}(u-b)^2\right)}{u} \right) du \right)}
\eqn{= a \cdot \Theta\left( \int_{1+2b}^{1+\sqrt{n}} \left( \frac{b \ln\left(\dfrac{(u-b)^2}{a}\right)}{(u - b) \cdot u} + \frac{2b}{u \cdot (u-b)} - \frac{S}{n} + \frac{\ln\left(\dfrac{(u-b)^2}{u^2}\right)}{u} + \frac{\ln\left(\dfrac{Su^2}{n}\right)}{u} \right) du \right)}
We now invoke Lemmas \ref{keysmash_b} and \ref{Snu5int} to approximately integrate the $S$-terms, yielding
\eqn{= a \cdot \Theta\left( \int_{1+2b}^{1+\sqrt{n}} \left( \frac{b \ln\left(\dfrac{(u-b)^2}{a}\right)}{(u - b) \cdot u} + \frac{2b}{u \cdot (u-b)} + \frac{\ln\left(\dfrac{(u-b)^2}{u^2}\right)}{u} \right) du + O(1/b) \right).}
The integral will turn out to be $\softTheta(1)$, so the $O(1/b)$ is neglectably small.
\eqn{= a \cdot \Theta\left( \int_{1+2b}^{1+\sqrt{n}} \left( \frac{b \ln\left(\dfrac{(u-b)^2}{a}\right)}{(u - b) \cdot u} + \frac{2b}{u \cdot (u-b)} + \frac{\ln\left(\dfrac{(u-b)^2}{u^2}\right)}{u} \right) du \right)}
Now we put the integrand on a common denominator and use the logarithm's quotient rule to obtain
\eqn{= a \cdot \Theta\left( \int_{1+2b}^{1+\sqrt{n}} \left( \frac{2b \ln(u-b) - b \ln(a) + 2b + 2(u-b) \ln(u-b) - 2(u-b) \ln(u)}{(u - b) \cdot u} \right) du \right).}
Some terms in the numerator cancel.
\eqn{= a \cdot \Theta\left( \int_{1+2b}^{1+\sqrt{n}} \left( \frac{ - b \ln(a) + 2b + 2u \ln(u-b) - 2(u-b) \ln(u)}{(u - b) \cdot u} \right) du \right)}
We now undo the common denominator to find that the integral has become elementary.
\eqn{= a \cdot \Theta\left( \int_{1+2b}^{1+\sqrt{n}} \left( (2-\ln(a))\left(\frac{1}{u-b}-\frac{1}{u}\right) + \frac{2\ln(u-b)}{u - b} - \frac{2\ln(u)}{u} \right) du \right)}
\eqn{= a \cdot \Theta\left( \eval{\left( (2-\ln(a))\left(\ln(\abs{u-b}) - \ln(\abs{u})\right) + (\ln(u-b))^2 - (\ln(u))^2 \right)}{u=1+2b}{1+\sqrt{n}}\right)}
\eqn{= a \cdot \Theta\left( \eval{\left( \ln(e^2 a^{-1})\ln\left(1-\frac{b}{u}\right) + \ln(u^2-bu)\ln\left(1-\frac{b}{u}\right) \right)}{u=1+2b}{1+\sqrt{n}}\right)}
\eqn{= a \cdot \Theta\left( \eval{\left( \ln\left(e^2 a^{-1}(u^2-bu)\right) \cdot \ln\left(1-\frac{b}{u}\right) \right)}{u=1+2b}{1+\sqrt{n}}\right)}
\eqn{= a \cdot \Theta\left( \ln\left(\frac{e^2}{a}(1+\sqrt{n})(1+\sqrt{n}-b)\right) \ln\left(1-\frac{b}{1+\sqrt{n}}\right) - \ln\left(\frac{e^2}{a}(2b^2+3b+1)\right) \ln\left(\frac{1+b}{1+2b}\right) \right)}
\eqn{= a \cdot \Theta\left( \ln\left(e^2 a^{-1} \cdot \Theta(n)\right) \cdot \ln\left(1-\Theta\left(\frac{b}{\sqrt{n}}\right)\right) + \ln\left(e^2 a^{-1} \cdot \Theta(b^2)\right) \cdot \Theta(\ln(2)) \right)}
Since we are working with $a = \softTheta(n^{2/3})$ and $b = \softTheta(n^{1/3})$, this becomes
\eqn{= a \cdot \Theta\left( \ln\left(e^2 \cdot \softTheta(n^{-2/3}) \cdot \Theta(n)\right) \cdot \ln\left(1-\softTheta(n^{-1/6})\right) + \ln\left(a^{-1} \cdot \Theta(b^2)\right) \cdot \Theta(1) \right)}
\eqn{= a \cdot \Theta\left( \ln\left(e^2 \cdot \softTheta(n^{1/3})\right) \cdot (-\softTheta(n^{-1/6})) + \ln\left(a^{-1} \cdot \Theta(b^2)\right) \cdot \Theta(1) \right)}
\eqn{= a \cdot \Theta\left( o(1) + \ln\left(\Theta(1) \cdot b^2/a\right) \cdot \Theta(1) \right)}
\eqn{= a \cdot \Theta\left( o(1) + \ln(\Theta(1)) \cdot \Theta(1) + \ln(b^2/a) \cdot \Theta(1) \right).}
Since we are working with $a = (n/\ln(\ln(n)))^{2/3}$ and $b = n^{1/3} \cdot (\ln(\ln(n)))^{2/3}$, $\ln(b^2/a)$ increases without bound.  Therefore,
\eqn{= a \cdot \Theta\left(\ln(b^2/a) \cdot \Theta(1)\right)}
\eqn{= \Theta\left(a \cdot \ln\left(n^2 a^{-3}\right)\right),}
as desired.
\end{proof}

\begin{lemma} \label{mfeqljk_u}
For $1 + b \leq u \leq 1 + \sqrt{n}$ and large $n$, $\dfrac{bn}{(u-b) \cdot u} > b$.
\end{lemma}
\begin{proof}
Since $\sqrt{n} > u-1$, we have
\eqn{n > (u-1)^2.}
For all but the smallest $n$, we will have $b > 2$; since we are working in the limit of large $n$, we can weaken this to
\eqn{n > (u-1)^2 - (bu - 2u + 1) = u^2 - bu,}
from which the conclusion follows immediately.
\end{proof}

\begin{lemma} \label{keysmash_b}
$\displaystyle \int_{1+2b}^{1+\sqrt{n}} \left( \frac{1}{u} \cdot \ln \left( u^2 \cdot \frac{S}{n} \right) \right) du = O(b^{-1})$.
\end{lemma}
\begin{proof}
Let $C_k$ be the $k$\textsuperscript{th} Catalan number (\seqnum{A000108}).  Observe that
\eqn{\frac{S}{n} = \sum_{k=2}^\infty \frac{C_{k-1}}{u^k},}
and therefore $S/n > u^{-2}$.  Then
\eqn{0 < \int_{1+2b}^{1+\sqrt{n}} \left(\frac{1}{u} \cdot \ln\left( u^2 \cdot \frac{S}{n} \right) \right) du.}
By Lemma \ref{sqrtapprox2}, we have $u^2 \cdot S/n = 1 + O(u^{-1})$, so
\eqn{ = \int_{1+2b}^{1+\sqrt{n}} \left(\frac{1}{u} \cdot \ln\left( 1 + O(u^{-1}) \right) \right) du}
\eqn{ = \int_{1+2b}^{1+\sqrt{n}} \left(\frac{1}{u} \cdot O(u^{-1}) \right) du}
\eqn{ = \int_{1+2b}^{1+\sqrt{n}} O(u^{-2}) \; du}
\eqn{ = \eval{-O(u^{-1})}{u=1+2b}{1+\sqrt{n}}}
\eqn{ = O\left(\frac{1}{1+2b}\right) - O\left(\frac{1}{1+\sqrt{n}}\right)}
\eqn{ = O(b^{-1}),}
as desired.
\end{proof}

\begin{lemma} \label{Snu5int}
$\disp \int_{1+2b}^{1+\sqrt{n}} \frac{S}{n} \; du = O(b^{-1})$.
\end{lemma}
\begin{proof}
For $u > 5$, we have
\eqn{\frac{16}{u^2} + \frac{16}{u} < 4}
\eqn{u^2 + 4 + \frac{16}{u^2} - 4u - 8 + \frac{16}{u} < u^2 - 4u}
\eqn{u - 2 - \frac{4}{u} < \sqrt{u^2 - 4u}}
\eqn{u - 2 - \sqrt{u^2 - 4u} < \frac{4}{u}}
\eqn{0 < u - 2 - \sqrt{u^2 - 4u} < \frac{4}{u}}
\eqn{0 < \frac{1}{2} - \frac{1}{u} - \sqrt{\frac{1}{4} - \frac{1}{u}} < \frac{2}{u^2}}
\eqn{0 < \frac{S}{n} < \frac{2}{u^2}.}
Integrating then gives the desired result.
\end{proof}


\setlength{\bibitemsep}{\parskip}
\printbibliography[heading=bibnumbered]

@article{HKM2024,
    title = {Computing $\pi(N)$: An elementary approach in $\tilde{O}(\sqrt{N})$ time},
    author = {Dean Hirsch and Ido Kessler and Uri Mendlovic},
    year = {2022},
    url = {https://www.ams.org/journals/mcom/0000-000-00/S0025-5718-2024-04039-5/},
    doi = {10.1090/mcom/4039},
    journal = {Mathematics of Computation},
    issn = {1088-6842},
    volume = {},
    number = {},
    pages = {},
    archivePrefix = {arXiv},
    eprint = {2212.09857},
    primaryClass = {math.NT},
}

@online{griff2023,
    title = {Summing Multiplicative Functions (Pt. 1)},
    author = {Griffin Macris},
    url = {https://gbroxey.github.io/blog/2023/04/30/mult-sum-1.html},
    year = {2023},
}

@online{adamant2023,
    title = {Dirichlet convolution.  Part 1: Fast prefix sum computations},
    author = {Oleksandr Kulkov},
    url = {https://codeforces.com/blog/entry/117635},
    year = {2023},
}

@article{DR1996,
    title = {Computing the summation of the M\"{o}bius function},
    author = {Marc Del\'{e}glise and Jo\"{o}l Rivat},
    year = {1996},
    url = {https://projecteuclid.org/euclid.em/1047565447},
    doi = {10.1080/10586458.1996.10504594},
    journal = {Experimental Mathematics},
    issn = {1944-950X},
    volume = {5},
    number = {4},
    pages = {291--295},
}

@article{HT23,
    title = {Summing $\mu(n)$: a faster elementary algorithm},
    author = {Harald Andr\'{e}s Helfgott and Lola Thompson},
    year = {2023},
    url = {https://link.springer.com/article/10.1007/s40993-022-00408-8},
    doi = {10.1007/s40993-022-00408-8},
    journal = {Research in Number Theory},
    issn = {2363-9555},
    volume = {9},
    number = {6},
    archivePrefix = {arXiv},
    eprint = {2101.08773v4},
    primaryClass = {math.NT},
}

@misc{HKM2025,
    author = {Dean Hirsch and Ido Kessler and Uri Mendlovic},
    date = {2025-05-01},
    howpublished = {Personal communication},
}

\bigskip
\hrule
\bigskip

\noindent 2020 \emph{Mathematics Subject Classification}: Primary 11Y16; Secondary 11-04, 11A25, 11Y55, 11Y70.

\noindent \emph{Keywords}: totient, totient summatory function, Dirichlet hyperbola method, computation.


\bigskip
\hrule
\bigskip

\noindent (Concerned with sequences \seqnum{A002088} and \seqnum{A064018}.)

\end{document}